\documentclass{amsart}
\usepackage{color}
\usepackage{verbatim}
\usepackage{amsthm}
\usepackage{amstext}
\usepackage{graphicx}
\usepackage{amssymb}

\makeatletter
\numberwithin{equation}{section} 
\numberwithin{figure}{section} 
\theoremstyle{plain}
\newtheorem{thm}{Theorem}[section]
  \theoremstyle{definition}
  \newtheorem{defn}[thm]{Definition}
  \theoremstyle{plain}
  \newtheorem{prop}[thm]{Proposition}
  \theoremstyle{plain}
  \newtheorem{lem}[thm]{Lemma}
  \theoremstyle{remark}
  \newtheorem{rem}[thm]{Remark}
\theoremstyle{remark}
\newtheorem*{acknowledgment*}{Acknowledgment}

\makeatother

\begin{document}

\title{Hausdorff Measures and KMS States}

\author{Marius Ionescu}

\address{Department of Mathematics,  Colgate University, 13 Oak Dr,
  Hamilton, NY 13346, USA}

\email{mionescu@colgate.edu}
\thanks{This work was partially supported by a grant from the Simons Foundation (\#209277 to Marius Ionescu).}

\author{Alex Kumjian}

\address{Department of Mathematics, University of Nevada Reno NV 89557 USA}

\email{alex@unr.edu}

\begin{abstract}
Given a compact metric space $X$ and a local homeomorphism $T:X\to X$
satisfying a local scaling property, we show that the Hausdorff measure
on $X$ gives rise to a KMS state on the $C^{*}$-algebra naturally
associated to the pair $(X,T)$ such that the inverse temperature
coincides with the Hausdorff dimension. We prove that the KMS state
is unique under some mild hypotheses. We use our results to describe
KMS states on Cuntz algebras, graph algebras, and certain $C^{*}$-algebras
associated to fractafolds.
\end{abstract}
\maketitle

\section{Introduction}

\global\long\def\Aut{\operatorname{Aut}}
This note arises from a simple observation which relates to the computation
of the Hausdorff dimension $\dim X$ of a self-similar space $X$ arising from an iterated function
system 
that satsifies the open set condition and has similarity ratios $r_{1},\dots,r_{n}\in(0,1)$;
$\dim X$ is the unique positive number $s$ satisfying\[
1=r_{1}^{s}+\dots+r_{n}^{s}\]
(see \cite[Section 6.4]{Edg_MTFG2}). But this agrees with the condition that ensures
that a certain one-parameter automorphism group on the Cuntz algebra 
 $\mathcal{O}_{n}$ has a KMS state. Let $\alpha:\mathbb{R}\to\Aut(\mathcal{O}_{n})$
be defined such that $\alpha_{t}(S_{j})=e^{it\lambda_{j}}S_{j}$ for
$j=1,2,\dots,n$; then by \cite[Proposition 2.2]{Eva_80} there is
a KMS state for $\alpha$ at inverse temperature $\beta$ iff\[
1=e^{-\beta\lambda_{1}}+\dots+e^{-\beta\lambda_{n}}.\]
If we set $r_{j}=e^{-\lambda_{j}}$ for $j=1,\dots,n$ and $s=\beta$,
the two conditions coincide (see Example \ref{sub:Cuntz-algebras}
for more details). It is the purpose of this note to account for this
coincidence and explore more examples. To accomplish this we apply
groupoid methods  from \cite{Ren_LNM793} and
\cite{Ku_Re_PAMS06} (see also \cite{Ren_2009}). 

In 1967 Haag, Hugenholtz and Winnink (see \cite{HaHuWi_CMP67}) discovered the 
relevance of  the KMS condition to the $C^*$-algebraic formulation of equilibrium states 
in quantum statistical mechanics.  Thereafter the KMS states of a $C^*$-algebra with 
respect to a natural one-parameter automorphism group (regarded as time evolution) 
have played a key role in the development of the theory.  Many of the basic facts 
may be found in the text by Bratteli and Robinson (see \cite{Br_Ro_OAQS2}).

Given a compact metric space $X$, a local homeomorphism $T:X\to X$ is said to satisfy
the local scaling condition (see Definition \ref{def:loc-scaling})
if $(x, y) \mapsto \frac{\rho(Tx,Ty)}{\rho(x,y)}$ extends to a continuous
function on $X \times X$ that is strictly positive on the diagonal.
In this case, there is a continuous function  
$\varphi\in C(X,\mathbb{R})$  such that for all $x\in X$,
\[
e^{\varphi(x)} = \lim_{y\to x}\frac{\rho(Tx,Ty)}{\rho(x,y)}.
\]
Then $\varphi$ gives rise to a continuous real-valued cocycle $c_\varphi$ on the 
Renault-Deaconu groupoid $G$ which in turn defines a one-parameter 
action $\alpha$ on $C^{*}(G)$, the associated groupoid $C^{*}$-algebra.

Our main result (see Theorem \ref{thm:Main}) asserts that if $T:X\to X$ satisfies the local scaling condition, 
then the Hausdorff measure may be used to define an $(\alpha, \beta)$-KMS state 
on $C^{*}(G)$ where $\beta$ is the Hausdorff dimension of $X$.    We next obtain conditions for the uniqueness of the KMS state (see Proposition \ref{pro:uniqueness}). 
If $\varphi$ is constant, we derive a simple equation involving the topological 
entropy of $T$, the Hausdorff dimension of $X$ and the scaling constant 
(see Proposition \ref{pro:Entropy}).

The remainder of the paper is devoted to applications of our results to a number 
of examples.  The last example, which is based on the Sierpinski octafold, is perhaps 
the most interesting. Although it does not satisfy the hypotheses of our main result, 
the conclusions hold.  This suggests that it should be possible to weaken the hypotheses in our main result.

\section{Preliminaries}\label{sec:prelim}
Let $(X,\rho)$ be a metric space and let $s>0$. For a set $F\subset X$
and $\varepsilon>0$, a countable cover $\mathcal{A}$ of $F$ is
called an $\varepsilon$-\emph{cover} of $F$ iff \global\long\def\diam{\operatorname{diam}}
 $\diam A\le\varepsilon$ for all $A\in\mathcal{A}$. Define\[
\overline{\mu}_{\varepsilon}^{s}(F)=\inf\sum_{A\in\mathcal{A}}(\diam A)^{s},\]
where the infimum is over all countable $\varepsilon$-covers $\mathcal{A}$
of the set $F$ \cite[Section 6.1]{Edg_MTFG2}. Then $\overline{\mu}_{\varepsilon}^{s}$
is decreasing with respect to $\varepsilon$ and\[
\overline{\mu}^{s}(F)=\lim_{\varepsilon\to0}\overline{\mu}_{\varepsilon}^{s}(F)\]
is a metric outer measure on $X$. Let $\mu^{s}$ be the Borel measure
defined by $\overline{\mu}^{s}$. 

Let $s,t>0$ such that $s<t$. It is well known (\cite[Theorem 6.1.6]{Edg_MTFG2})
that if $\mu^{t}(F)>0$ then $\mu^{s}(F)=\infty$ and if $\mu^{s}(F)<\infty$
then $\mu^{t}(F)=0$. The \emph{Hausdorff dimension }$\dim F$ of
a set $F$ is the unique number $s_{0}\in[0,\infty]$ such that $\mu^{s}(F)=\infty$
for all $s<s_{0}$ and $\mu^{s}(F)=0$ for all $s>s_{0}$. If $s=\dim X$,
then we call $\mu^{s}$ the \emph{Hausdorff measure} on $X$. In the
following we assume that $0<s<\infty$ and $0<\mu^{s}(X)<\infty$.

We say $G$ is a groupoid
  (\cite{Ren_LNM793}) if there is a
subset $G^{(2)}$ of $G\times G$, a map $(x,y)\mapsto xy$ from
$G^{(2)}$ to $G$ and a involution $x\mapsto x^{-1}$ on $G$ such that
the following conditions hold:
\begin{enumerate}
\item If $(x,y)$ and $(y,z)$ are in $G^{(2)}$, then so are $(xy,z)$
  and $(x,yz)$, and the equation $(xy)z=x(yz)$ is satisfied;
\item For all $x\in G$, $(x^{-1},x)\in G^{(2)}$ and if $(x,y)\in
  G^{(2)}$, then $x^{-1}(xy)=y$ and $(xy)y^{-1}=x$.
\end{enumerate}
The maps $r$ and $s$ on $G$, defined by the formulae $r(x)=xx^{-1}$
and $s(x)=x^{-1}x$ are called the \emph{range} and \emph{source}
maps. Then $G^{(0)}:=r(G)=s(G)$ is called the \emph{unit space} of
$G$. The groupoid $G$ is called a topological groupoid in case $G$
  is a groupoid with a topology such that the
multiplication and the inverse maps are continuous, where the topology
on $G^{(2)}$ is the relative product topology. The groupoids that we
consider in this note are Hausdorff and locally compact. 

A left \emph{Haar system} (\cite{Ren_LNM793}) on a topological groupoid $G$ is a family
  $\{\lambda^u\}_{u\in G^{(0)}}$ of non-negative Radon measures on $G$
  such that
\begin{enumerate}
\item $\operatorname{supp}(\lambda^u)=G^u$, $u\in G^{(0)}$;
\item for $f\in C_c(G)$, the function
    \[
    u\mapsto \int fd\lambda^u
    \]
    on $G^{(0)}$ is in $C_c(G^{(0)})$ ; and
  \item for $x\in G$, $\int f(xy)d\lambda^{s(x)}(y)=\int f(y)d\lambda^{r(x)}$.
\end{enumerate}
Once a Haar system $\{\lambda^y\}_{u\in G^{(0)}}$ has been
  specified on a groupoid, one can define an involutive algebraic
  structure on $C_c(G)$ by the formulae
\[
f\ast g(y)=\int f(x)g(x^{-1}y)d\lambda^{r(y)}
\]
and
\[
f^\ast(x)=\overline{f(x^{-1})}.
\]
The groupoid $C^\ast$-algebra $C^\ast(G)$ is the completion of
$C_c(G)$ under the universal norm (see \cite{Ren_LNM793} for details).

Assume that $T:X\to X$ is a local homeomorphism. The \emph{Renault-Deaconu}
groupoid associated to the pair $(X,T)$ (see \cite{Ren_CLA00}, \cite{Dea_TAMS95}) is defined via\[
G:=\{(x,m-n,y)\,:\, T^{m}(x)=T^{n}(y)\}\subset X\times\mathbb{Z}\times X.\]
Two triples $(x_{1},n_{1},y_{1})$ and $(x_{2},n_{2},y_{2})$ are
composable if and only if $y_{1}=x_{2}$, and, in this case, 
$(x_{1},n_{1},y_{1})(y_{1},n_{2},y_{2})=(x_{1},n_{1}+n_{2},y_{2})$.
The inverse of $(x,n,y)$ is $(y,-n,x)$. A basis of topology for
$G$ is given by the sets\[
Z(U,V,k,l)=\{(x,k-l,y)\in G\,:\, x\in U,y\in V\},\]
where $U$ and $V$ are open subsets of $X$ such that $T^{k}|_{U}$
and $T^{l}|_{V}$ are homeomorphisms onto their images and $T^{k}(U)=T^{l}(V)$.
The groupoid $G$ is \'etale so it admits a Haar system consisting
of counting measures.

Let $A$ be a $C^{*}$-algebra, let $\alpha_{t}:\mathbb{R}\to\Aut(A)$
be a strongly continuous action, and let $\beta\in\mathbb{R}$. Recall
(\cite[Section 5.3]{Br_Ro_OAQS2}) that a state $\varphi$ on $A$
is said to be an $(\alpha,\beta)$-KMS state if\[
\varphi(b\alpha_{i\beta}(a))=\varphi(ab)\]
for all $a,b$ entire for $\alpha$. If $\beta=0$, then $\varphi$
is an $\alpha$-invariant tracial state. The parameter $\beta$ is
called the \emph{inverse temperature}. 

Let $G$ be a locally compact \'etale groupoid and let $c$ be a
real-valued continuous cocycle. Then $c$ defines a one-parameter
automorphism group $\alpha^{c}$ of $C^{*}(G)$ via\[
\alpha_{t}^{c}(f)(\gamma)=e^{itc(\gamma)}f(\gamma),\]
for all $t\in\mathbb{R}$, $\gamma\in G$, and $f\in C_{c}(G)$ (\cite[Section II.5]{Ren_LNM793}).
Each probability measure $\mu$ on $G^{(0)}$ defines a state $\omega_{\mu}$
on $C^{*}(G)$ via\begin{equation}
\omega_{\mu}(f)=\int_{G^{(0)}}E(f)d\mu,\label{eq:KMSstate}\end{equation}
for all $f\in C^{*}(G)$, where $E$ is the canonical expectation
onto $C_{0}(G^{(0)})$. A probability measure $\mu$ on $G^{(0)}$
is quasi-invariant under $G$ with Radon-Nikod\'ym derivative $dr^{*}\mu/ds^{*}s=e^{-\beta c}$
if and only if the state $\omega_{\mu}$ of $C^{*}(G)$ is an $(\alpha^{c},\beta)$-KMS
state (\cite[Proposition II.5.4]{Ren_LNM793}). Moreover, if $c^{-1}(0)$
is principal then every $(\alpha^{c},\beta)$-KMS state of $C^{*}(G)$
is of the form $\omega_{\mu}$ for some quasi-invariant probability
measure $\mu$ on $G^{(0)}$ with Radon-Nikod\'ym derivative $dr^{*}\mu/ds^{*}\mu=e^{-\beta c}$
(\cite[Proposition 3.2]{Ku_Re_PAMS06}).

If $G$ is the Renault-Deaconu groupoid associated with a local homeomorphism
$T$ on a compact metric space $(X,\rho)$, then every real valued
continuous function $\varphi$ on $X$ defines a continuous one-cocycle
on $G$ via the formula\[
c_{\varphi}(x,m-n,y)=\sum_{i=0}^{m-1}\varphi(T^{i}x)-\sum_{i=0}^{n-1}\varphi(T^{i}y).\]
Moreover, every continuous one-cocycle on $G$ is of this form (see
\cite[Lemma 2.1]{De_Ku_Mu_JOT01}). In this case, the condition that
$c_{\varphi}^{-1}(0)$ is principal means that (\cite[page 2073]{Ku_Re_PAMS06})\[
T^{n}(x)=x\;\text{and}\: n\ge1\;\implies\;\varphi(x)+\varphi(Tx)+\dots+\varphi(T^{n-1}x)\ne0.\]
This condition is trivially satisfied when $\varphi$ is strictly
positive or strictly negative.

\section{The Main Result}
In this section we show that if $(X,\rho)$ is a compact metric space
and $T:X\to X$ is a local homeomorphism which satisfies additional
conditions as described below, then the Hausdorff measure gives rise
to a KMS-state with inverse temperature the Hausdorff dimension. We
will also show that the KMS state is unique under some mild hypotheses.
\begin{defn}\label{def:loc-scaling}
Let $(X,\rho)$ be a compact metric space and let $T:X\to X$ be a
local homeomorphism. We say that $T$ satisfies the \emph{local scaling
condition} if 
\[
(x, y) \mapsto \frac{\rho(Tx,Ty)}{\rho(x,y)}
\] 
extends to a continuous function $f$ on $X \times X$ that is strictly
positive on the diagonal $\Delta_X = \{(x,x)\;|\; x\in X\}$. 
\end{defn}
The next proposition provides an equivalent formulation of the local
scaling property. 
\begin{prop}\label{rem:equivdef}
Let $(X,\rho)$ be a compact metric space and let $T:X\to X$ be a
local homeomorphism. 
Then $T$ satisfies the local scaling condition if and only if 
there is a continuous function $\varphi\in C(X,\mathbb{R})$ 
such that:
\begin{itemize}
\item[i.]
for all $x\in X$,
\[
e^{\varphi(x)} = \lim_{y\to x}\frac{\rho(Tx,Ty)}{\rho(x,y)}; 
\]
\item[ii.]
for any $\varepsilon>0$ there is $\delta>0$
such that if $0<\rho(x,y)<\delta$, then
\begin{equation}
\left\vert \frac{\rho(x,y)}{\rho(Tx,Ty)}-e^{-\varphi(x)}\right\vert <\varepsilon.  \label{eq:0}
\end{equation}
\end{itemize}
\end{prop}
\begin{rem}
  Notice that since $T$ is assumed to be a local homeomorphism and $X$
  is compact, there is $\delta_1>0$ such that if
  $0<\rho(x,y)<\delta_1$ then $Tx\ne Ty$. In the following we will
 implicitly assume that $\delta$ above 
satisfies $\delta \le \delta_1$.  This will ensure
  that inequality (\ref{eq:0}) makes sense.
\end{rem}

\begin{proof}
 Suppose that $f$ satisfies the local scaling condition.  
 Then the function $\varphi: X \to  \mathbb{R}$,
 given by  $\varphi(x)=\log f(x,x)$ for $x\in X$,
is continuous on $X$; since $f$ is uniformly continuous on
$X\times X$, it follows easily that the conditions above are satisfied.

Conversely, suppose that $T$ satisfies the above conditions.  We then define
$f:X\times X\mapsto \mathbb{R}$ via 
\[
f(x,y)=\begin{cases}\frac{\rho(Tx,Ty)}{\rho(x,y)}& \text{ if
      }x\ne y \\
      e^{\varphi(x)} & \text{ if }x=y.
    \end{cases}
    \]
Then $f(x,x) > 0$ for all $x\in X$ and it suffices to
show that $f$ is continuous. Since $f$ is clearly continuous at $(x,y)$ if
$x\ne y$, we need only prove that $f$ is continuous on the diagonal.

We prove first that $f$ is bounded on $X\times X$. Assume, by
contradiction, that for each $n\in\mathbb{N}$ there are $x_n$ and
$y_n$ in $X$
such that $\vert f(x_n,y_n)\vert >n$.  Then there is a subsequence
$(x_{n_k},y_{n_k})$ that converges to $(x,y)\in X\times X$. If $x\ne y$
then it follows that $x_{n_k}\ne y_{n_k}$ eventually, and that $\lim_{k\to
  \infty}f(x_{n_k},y_{n_k})=f(x,y)$. 
But this contradicts the fact that the sequence $|f(x_{n_k},y_{n_k})|$
is unbounded.
Assume now that
$x=y$. Since $\varphi$ is continuous on $X$ and thus bounded, 
we may assume (by passing to a subsequence if necessary) that $x_{n_k}\ne y_{n_k}$ for all $k$.
Given $\varepsilon>0$, choose $\delta>0$ such that (\ref{eq:0})
holds. 
Then there is an integer $N \ge 1$ such that   for all $k \ge N$
\begin{alignat*}{3}
\rho(x,x_{n_k})&<\delta/2,\  & \rho(x,y_{n_k})&<\delta/2 &\qquad\text{and}\\
&  \vert 1/f(x, x)&-1/f(x_{n_k},x_{n_k})\vert &< \varepsilon
\end{alignat*}
(by the continuity of $e^{-\varphi}$). 
It follows that for $k\ge N$ 
$\rho(x_{n_k},y_{n_k})<\delta$ and, since
  $e^{-\varphi(x)}=1/f(x,x)$, we also have that 
\[
\left\vert \frac{1}{f(x_{n_k},y_{n_k})}-\frac{1}{f(x_{n_k},x_{n_k})}
\right\vert<\varepsilon .
\]
Therefore
\begin{align*}
\left\vert \frac{1}{f(x_{n_k},y_{n_k})}-\frac{1}{f(x,x)}  \right\vert &\le
\left\vert \frac{1}{f(x_{n_k},y_{n_k})}-\frac{1}{f(x_{n_k},x_{n_k})}  \right\vert\ +
\left\vert \frac{1}{f(x_{n_k},x_{n_k})}-\frac{1}{f(x,x)}  \right\vert\
\\
 &< 2\varepsilon.
\end{align*}
Since $\vert 1/f(x_n,y_n)\vert <1/n$ it follows that $1/f(x,x)=0$,
which is a contradiction. 
Thus $f$ is bounded on $X\times X$; let $M>0$ be a bound
of the function.  

We return now to the continuity of $f$ for $x=y$. Let $x_n\to x$ and
$y_n\to x$. If $x_n=y_n$ for all $n$ then
$\lim_{n\to \infty}f(x_n,x_n)=f(x,x)$ because $\varphi$ is assumed to be continuous. So,
by passing  to a subsequence, if necessary,
we may assume that $x_n\ne y_n$ for all $n$. Let $\varepsilon>0$ be given;
choose $\delta>0$ such that (\ref{eq:0}) holds. 
There is an integer $N$ such that for all $n\ge N$ 
\begin{alignat*}{3}
\rho(x,x_n) &<\delta/2, & \rho(x,y_n)&<\delta/2  &\qquad\text{and}\\
&\vert 1/f(x,x)&-1/f(x_{n},x_{n})\vert&< \varepsilon.
\end{alignat*}
Then, for such $n$,  
\begin{align*}
  \vert f(x_n,y_n)-f(x,x)\vert &= \vert f(x_n,y_n)f(x,x)\vert \left\vert \frac{1}{f(x,x)} -
  \frac{1}{f(x_n,y_n)}\right\vert \\
 &\le M^2\left(\left\vert \frac{1}{f(x,x)}-\frac{1}{f(x_n,x_n)} \right\vert
 + \left\vert \frac{1}{f(x_n,x_n)} -\frac{1}{f(x_n,y_n)}\right\vert\right)\\
&< 2M^2\varepsilon.
\end{align*}
Therefore, $f$ is continuous on $X\times X$ and $T$ satisfies the local scaling condition. 
\end{proof}

\begin{thm}
\label{thm:Main}Let $(X,\rho)$ be a compact metric space and let
$T:X\to X$ be a local homeomorphism that satisfies the local scaling
condition. Let $\varphi\in C(X,\mathbb{R})$ be such that\[
e^{\varphi(x)}=\lim_{y\to x}\frac{\rho(Tx,Ty)}{\rho(x,y)},\]
and let $c_{\varphi}$ be the associated one-cocycle on the
  Renault-Deaconu groupoid $G$ associated with $(X,T)$. Then
the state $\omega_{\mu}$ given by the equation \eqref{eq:KMSstate}
is an $(\alpha,\beta)$-KMS state where $\alpha=\alpha^{c_{\varphi}}$,
$\beta=\dim X$, and $\mu=\mu^{\beta}/\mu^{\beta}(X)$.\end{thm}
\begin{proof}
As
noted above in Section \ref{sec:prelim},  by 
\cite[Proposition II.5.4]{Ren_LNM793}, it suffices to show that $\mu$ 
is a quasi-invariant measure on $G^{(0)}$ and
\begin{equation}
\frac{dr^{*}\mu}{ds^{*}\mu}=e^{-\beta c_{\varphi}}.  \label{eq:RN}
\end{equation}
This is equivalent with\begin{equation}
\mu(U)=\int_{TU}e^{-\beta\varphi\bigl((T|_{U})^{-1}x\bigr)}d\mu(x)\label{eq:qinv}\end{equation}
for all open sets $U$ such that $T|_{U}$ is a homeomorphism onto
$TU$ (\cite[Page 6]{Ren_ETDN_05}; see also \cite{Ren_OAMP_03}).
We will call such sets \emph{sections} of $T$.

By the local scaling property (see Definition \ref{def:loc-scaling}) 
$(x, y) \mapsto \frac{\rho(Tx,Ty)}{\rho(x,y)}$ extends to a continuous 
function $f$ on $X \times X$ that is strictly positive on the
  diagonal. We know that $\varphi(x)=\log f(x,x)$ is uniformly
  continuous, since $f$ is continuous and $X$ is compact. This fact
  together with \eqref{eq:0} implies that for any $\varepsilon>0$ there is $\delta>0$
such that if $0<\rho(x,y)<\delta$ then\begin{equation}
\left\vert \frac{\rho(x,y)}{\rho(Tx,Ty)}-e^{-\varphi(x)}\right\vert <\varepsilon\label{eq:1}\end{equation}
and\begin{equation}
\vert
e^{-\varphi(x)}-e^{-\varphi(y)}\vert<\varepsilon.\label{eq:2}\end{equation}
Fix $\varepsilon>0$ such that $2\varepsilon<\min_{x\in X}e^{-\varphi(x)}$.
Suppose, first, that $U$ is a section of $T$ such that $\diam U<\delta$,
where $\delta$ is chosen such that \eqref{eq:1} and \eqref{eq:2}
hold. Let $\underline{x},\overline{x}\in\overline{U}$ such that $e^{-\varphi(\underline{x})}\le e^{-\varphi(x)}\le e^{-\varphi(\overline{x})}$
for all $x\in U$, where $\overline{U}$ is the closure of the set
  $U$. Let $\varepsilon_{1}>0$ and suppose that $\mathcal{A}$ 
is an $\varepsilon_{1}$-cover of $U$ consisting of subsets of $U$
(see \cite[page 167]{Edg_MTFG2}). Then, for any $A\in\mathcal{A}$
we have that\[
\diam A=\sup_{x,y\in A,x\ne y}\rho(x,y)=\sup_{x,y\in A,x\ne y}\rho(Tx,Ty)\frac{\rho(x,y)}{\rho(Tx,Ty)}.\]
Now\[
\left\vert \frac{\rho(x,y)}{\rho(Tx,Ty)}-e^{-\varphi(\underline{x})}\right\vert <2\varepsilon\]
and\[
\left\vert \frac{\rho(x,y)}{\rho(Tx,Ty)}-e^{-\varphi(\overline{x})}\right\vert <2\varepsilon\]
for all $x,y\in A$ with $x\ne y$. Therefore\[
\bigl(e^{-\varphi(\overline{x})}-2\varepsilon\bigr)\diam TA\le\diam A\le\bigl(e^{-\varphi(\underline{x})}+2\varepsilon)\diam TA.\]
Hence \[
\bigl(e^{-\varphi(\overline{x})}-2\varepsilon\bigr)^{\beta}\overline{\mu}_{\varepsilon_{1}}^{\beta}(TU)\le\overline{\mu}_{\varepsilon_{1}}^{\beta}(U)\le\bigl(e^{-\varphi(\underline{x})}+2\varepsilon\bigr)^{\beta}\overline{\mu}_{\varepsilon_{1}}^{\beta}(TU)\]
and, by taking the limit when $\varepsilon_{1}$ goes to $0$,\begin{equation}
\bigl(e^{-\varphi(\overline{x})}-2\varepsilon\bigr)^{\beta}\mu(TU)\le\mu(U)\le\bigl(e^{-\varphi(\underline{x})}+2\varepsilon\bigr)^{\beta}\mu(TU).\label{eq:3}\end{equation}
Let now $U$ be an arbitrary section of $T$. Then there is a finite disjoint
family $\{U_{1}, \dots ,  U_{N}\}$ of sections of $T$ such that $\overline{U}=\bigcup_{n=1}^N\overline{U}_{n}$
and $\diam U_{n}<\delta$ for $n = 1, \dots, N$.   Indeed, by the compactness of $\overline{U}$, there are finitely
many open sets $\{V_{1}, \dots ,  V_{N}\}$ such that $\overline{U} \subset \bigcup_{n=1}^N V_{n}$ and
 $\diam V_{n}<\delta$ for $n = 1, \dots, N$.   Now set $U_1 =: U \cap V_1$ and 
 \[
 U_{n+1} =: U \cap V_{n+1} \setminus (\overline{V}_{1} \cup \cdots \cup \overline{V}_{n}),
 \]
 for $n\ge 1$.
It is routine to check that the $U_n$ satisfy the requisite conditions.
Let $\underline{x}_{n},\overline{x}_{n}\in\overline{U}_{n}$
such that $e^{-\varphi(\underline{x}_{n})}\le e^{-\varphi(x)}\le e^{-\varphi(\overline{x}_{n})}$
for all $x\in U_{n}$. Then, by inequality \eqref{eq:3},\[
\mu(U)=\sum_{n}\mu(U_{n})\le\sum_{n}\int_{TU_{n}}\bigl(e^{-\varphi(\underline{x}_{n})}+2\varepsilon)^{\beta}d\mu(x).\]
For all $x\in TU_{n}$ we have that $e^{-\varphi(\underline{x}_{n})}\le e^{-\varphi\bigl((T|_{U})^{-1}(x)\bigr)}$.
Thus\[
\mu(U)\le\int_{TU}\left(e^{-\varphi\bigl((T|_{U})^{-1}(x)\bigr)}+2\varepsilon\right)^{\beta}d\mu(x).\]
Similarly, since $e^{-\varphi\bigl((T|_{U})^{-1}(x)\bigr)}\le e^{-\varphi(\overline{x}_{n})}$
for all $x\in TU_{n}$, we have that\[
\int_{TU}\left(e^{-\varphi\bigl(T|_{U})^{-1}(x)\bigr)}-2\varepsilon\right)^{\beta}d\mu(x)\le\mu(U).\]
Since $\varepsilon$ was arbitrarily chosen it follows that \eqref{eq:qinv}
holds and the conclusion follows.
\end{proof}
To ensure the uniqueness of the $(\alpha,\beta)$-KMS state we need
to impose some conditions on the local homeomorphism $T$ and the
map $\varphi$. We say that $T$ is \emph{positively expansive} if
there is an $\varepsilon>0$ such that for all $x\ne y$ there is
an $n\in\mathbb{N}$ with $\rho(T^{n}x,T^{n}y)\ge\varepsilon$. We
say that $T$ is \emph{exact} if for every non-empty open set $U\subset X$
there is an $n>0$ such that $T^{n}(U)=X$.
We say that a real-valued continuous function $\varphi$ on $X$ satisfies the 
\emph{Bowen condition} with respect to $T$ 
(see \cite[Definition 2.7]{Ku_Re_PAMS06})
if there are $\delta, C>0$ such that
\[
\sum_{i=0}^{n-1}\varphi(T^{i}x)-\varphi(T^{i}y)\le C,
\]
for all $x, y \in X$ and $n > 0$ such that $\rho(T^ix, T^iy) \le \delta$
for $0 \le i \le n-1$.
Note that if $T$ is positively expansive and 
$\varphi$ is H\"older, that is, $\vert\varphi(x)-\varphi(y)\vert\le k\rho(x,y)^{l}$
for some positive constants $k$ and $l$, then $\varphi$ satisfies
the Bowen condition 
(see the discussion following \cite[Definition 2.7]{Ku_Re_PAMS06}).
\begin{prop}
\label{pro:uniqueness}Assume that the local homeomorphism $T:X\to X$
is positively expansive and exact, and assume that $\varphi$ satisfies
the Bowen condition. Let $\alpha$ be the action on $C^{*}(G)$ determined
by $c_{\varphi}$ and let $\beta$ be the Hausdorff dimension of $X$.
If $T$ satisfies the local scaling condition and $c_{\varphi}^{-1}(0)$
is principal then $\beta$ is the unique inverse temperature which
admits a KMS state for $\alpha$. Moreover, the $(\alpha,\beta)$-KMS
state $\omega_{\mu}$ is unique.\end{prop}
\begin{proof}
Theorem \ref{thm:Main} implies that there is a $(\alpha,\beta)$-KMS
state, namely $\omega_{\mu}$. By Walters\textquoteright{} version
of the Ruelle-Perron-Frobenius Theorem (\cite[Theorem 2.8]{Ku_Re_PAMS06},
see also \cite[Theorem 6.1]{Ren_ETDN_05}, \cite[Theorem 8]{Wal_78},
\cite[Theorem 2.16]{Wal_01}) there are unique $\lambda>0$ and probability
measure $\nu$ such that\[
\mathcal{L}_{-\varphi}^{*}(\nu)=\lambda\nu,\]
where $\mathcal{L}_{-\varphi}(f)(x)=\sum_{Ty=x}e^{-\varphi(y)}f(y)$
is the transfer operator associated with $\varphi$. Proposition 4.2
of \cite{Ren_OAMP_03} (see also \cite[Proposition 3.4.1]{Ren_2009})
implies that $\beta$ is the unique inverse temperature which admits
a KMS state for $\alpha$ and $\nu=\mu$. 

Theorem \ref{thm:Main} and \cite[Theorem 3.5 i)]{Ku_Re_PAMS06} imply
that $P(T,-\beta\varphi)=0$, where $P(T,\cdot)$ is the topological
pressure(\cite[Section 9.1]{Wal_GTM82}). Since $c_{\varphi}^{-1}(0)$
is principal, \cite[Theorem 3.5 ii)]{Ku_Re_PAMS06} implies that the
$(\alpha,\beta)$-KMS state $\omega_{\mu}$ is unique.
\end{proof}
In the last result of this section we show how one can compute the
topological entropy of the local endomorphism $T$ under some suitable
hypothesis. Recall that the topological entropy of $T$, $h(T)$,
equals $P(T,0)$.
\begin{prop}
\label{pro:Entropy}Let $T:X\to X$ be a local homeomorphism. Suppose
that $T$ is positively expansive, exact, and satisfies the local
scaling condition. Assume that the map $\varphi$ is constant,
that is, there is a constant $\tau>1$ such that\[
\tau=\lim_{y\to x}\frac{\rho(Tx,Ty)}{\rho(x,y)},\]
for all $x\in X$. Then $h(T)=\beta\ln\tau$, where $\beta$ is the
Hausdorff dimension of $X$.\end{prop}
\begin{proof}
By Theorem \ref{thm:Main} and Proposition \ref{pro:uniqueness} there
is a unique $(\alpha,\beta)$-KMS state $\omega_{\mu}$. Moreover,
Theorem 3.5 of \cite{Ku_Re_PAMS06} implies that $P(T,-\beta\ln\tau)=0$.
Since $P(T,-\beta\ln\tau)=h(T)-\beta\ln\tau$, it follows that $h(T)=\beta\ln\tau$.
\end{proof}

\section{Examples}

\subsection{\label{sub:Cuntz-algebras}Cuntz algebras}

Fix $n\in\mathbb{N}$ with $n>1$ and let $E=\{1,\dots,n\}$ be the
alphabet with $n$ letters. Let $(r_{1},r_{2},\dots,r_{n})$ be a
list of positive numbers such that $r_{i}<1$ for all $i\in\{1,\dots,n\}$.
We call such a list a \emph{contractive ratio list} (see \cite[Chapter 4]{Edg_MTFG2}).
We set $X$ to be the infinite path space over the alphabet $E$,
that is\[
X=E^{\infty}=\{(x_{k})_{k\in\mathbb{N}}\,:\, x_{k}\in E\}.\]
We write $E^{k}$ for the set of paths of length $k$ over $E$ and
we set $E^{*}=\bigcup_{k}E^{k}$. We define a metric on $X$ based
on the given ratio list. For $\sigma\in E^{k}$ the cylinder of $\sigma$
is\begin{equation}
Z(\sigma)=\{(x_{m})\in X\,:\, x_{0}=\sigma_{0},\dots,x_{k-1}=\sigma_{k-1}\}.\label{eq:cylinder}\end{equation}
We specify the metric on $X$ by requiring that the diameter of a
cylinder $Z(\sigma)$, with $\sigma\in E^{k}$ for some k, to equal
$r_{\sigma}:=r_{\sigma_{0}}r_{\sigma_{1}}\dots r_{\sigma_{k-1}}$ (see \cite[Section 4.2]{Edg_MTFG2}). Thus
for $x\ne y$\begin{eqnarray*}
\rho(x,y) & = & \begin{cases}
\inf\{r_{\sigma}\,:\, x,y\in Z(\sigma)\} & \text{if }x_{0}=y_{0}\\
1 & \text{if }x_{0}\ne y_{0}\end{cases}.\end{eqnarray*}
Then the left shift, $T:X\to X$, $T((x_{k}))=(x_{k+1})$, is a local
homeomorphism on $X$. If $G$ is the corresponding Renault-Deaconu
groupoid then $C^{*}(G)$ is isomorphic to the Cuntz algebra $\mathcal{O}_{n}$
(\cite[Section III.2]{Ren_LNM793}, \cite{Dea_TAMS95}). Recall that
  the Cuntz algebra
  is the universal $C^\ast$-algebra generated by $n$ isometries $\{S_i\}_{i=1}^n$ satisfying
  the following relations
\[
S_i^\ast S_j=\delta_{ij}I\, \text{and}\, \sum_{i=1}^n S_iS_i^\ast =I.
\]
Under our identification
the generating isometries of $\mathcal{O}_{n}$ are given by 
$S_{j}=1_{\Gamma_{j}}$, where $\Gamma_{j}=\{(jx,1,x)\,:\, x\in X\}$, 
for $j=1,\dots,n$.    A routine computation shows that $T$ satisfies the 
local scaling property and 
\[
\lim_{y\to x}\frac{\rho(Tx,Ty)}{\rho(x,y)}=\frac{1}{r_{x_{0}}},
\]
that is, $\varphi(x)=-\log r_{x_{0}}$. Thus the automorphism
$\alpha$ on $\mathcal{O}_{n}$ is determined by
\begin{equation}
\alpha_{t}(S_{j})=e^{-it\log r_{j}}S_{j}.\label{eq:gauge_action}
\end{equation}
The Hausdorff dimension $\beta$ of $X$ is the
unique number that satisfies the equation (\cite[Theorem 6.4.3]{Edg_MTFG2})
\begin{equation}
\sum_{i=1}^{n}r_{i}^{s}=1.\label{eq:Hausdorff_Cantor}\end{equation}
The Hausdorff measure on $X$ is the unique Borel measure that satisfies
$\mu^{\beta}(Z(\sigma))=r_{\sigma}^{\beta}$ for all $\sigma\in E^{*}$.
Theorem \ref{thm:Main} and Proposition \ref{pro:uniqueness} imply
that there exists a KMS-state for $\alpha$ at temperature $\beta$
if and only if $\beta$ satisfies \eqref{eq:Hausdorff_Cantor}. The
KMS state is unique in this case. We recover, thus, Theorem 2.2 of
\cite{Eva_80}. 

If $r_{1}=r_{2}=\dots=r_{n}=1/e$, then the inverse temperature is
$\log n$, which is the main result in \cite{Ole_Ped_MS78}. Moreover
the topological entropy of $T$ is $\log n$. More generally, let
$s>0$ be arbitrary and set $r_{1}=r_{2}=\dots=r_{n}=n^{-1/s}$. Then
$(r_{1},\dots,r_{n})$ is a contractive ratio list. The Hausdorff
dimension of $X$ is $\beta=s$ and $\varphi(x)=\frac{\log n}{s}$.
The topological entropy of $T$ is still $\log n$.

\subsection{\label{sub:Generalized-gauge-actions-Cuntz}Generalized gauge actions
on Cuntz algebras}

We still assume that $E=\{1,\dots,n\}$ is a finite alphabet and we
let $X=E^{\infty}$ endowed with the product topology. Then $X$ is
a compact topological space. Let $T$ be the left shift, as in the
previous example. Fix now a continuous function $f$ on $X$ such
that $f(x)>1$ for all $x\in X$. We  define next a metric $\rho_{f}$
on $X$ with the property
\begin{equation}
\lim_{y\to
  x}\frac{\rho_{f}(Tx,Ty)}{\rho_{f}(x,y)}=f(x).\label{eq:f_metric}
\end{equation}
Let $\sigma\in E^{*}$ and let $Z(\sigma)$ be the cylinder of $\sigma$
as defined in \eqref{eq:cylinder}. We will write $\vert\sigma\vert$
for the length of the finite word $\sigma$. Define
\begin{equation}\label{eq:w_sigma}
w_{\sigma}:=\max_{z,w\in
  Z(\sigma)}\left(\prod_{i=0}^{\vert\sigma\vert-1}f(T^{i}z)\cdot\prod_{i=0}^{\vert\sigma\vert-1}f(T^{i}w)\right)^{-1/2}.
\end{equation}
Let $\alpha,\beta\in E^{*}$ such that $\vert\alpha\vert<\vert\beta\vert$
and $\alpha_{i}=\beta_{i}$ for $i=0,\dots,\vert\alpha\vert-1$; in this case,
write  $\alpha<\beta$ (\cite{Edg_MTFG2}). Since $f(x)>1$
for all $x\in X$, it follows that if $\alpha<\beta$ then
$w_{\beta}>w_{\alpha}$ and $\lim_{n\to \infty}w_{\sigma|n}=0$ for
$\sigma\in E^\infty$.
Proposition 2.6.5 of \cite{Edg_MTFG2} implies that there is a metric
$\rho_{f}$ on $X$ such that the diameter of $Z(\alpha)=w_{\alpha}$
for all $\alpha\in E^{*}$ and, if $x,y\in X$ such that
$x_{i}=y_{i}$ for $i=0,\dots,k-1$ and $x_{k}\ne y_{k}$ then
$\rho_{f}(x,y)=w_\sigma$, where $\sigma\in E^k$ such that $\sigma_i=x_i$, 
$i=0,\dots, k-1$. Thus if $x,y\in X$ and the length of their common
longest prefix $\sigma$ is at least one, then
\begin{equation}
  \label{eq:4}
  \rho_f(x,y)=w_\sigma:=\max_{z,w\in
  Z(\sigma)}\left(\prod_{i=0}^{\vert\sigma\vert-1}f(T^{i}z)\cdot\prod_{i=0}^{\vert\sigma\vert-1}f(T^{i}w)\right)^{-1/2}.
\end{equation}
Note that the maximum in (\ref{eq:4}) is attained for some
$\overline{z}$ and $\overline{w}$ in $Z(\sigma)$ since the function
$f$ is continuous and the cylinder $Z(\sigma)$ is a closed subset
  of the compact space $X$. Also, if $m=\min_{x\in X}\vert f(x)\vert$ and
$M=\max_{x\in X} \vert f(x)\vert$, then
\[
\frac{1}{M^{\vert\sigma\vert}}\le \rho_f(x,y) \le \frac{1}{m^{\vert\sigma\vert}},
\]
where $\sigma$ is the longest common prefix of $x$ and $y$, for
  all $x$ and $y$ in $X$. Thus,
since $m>1$ by hypothesis, for $x$ and $y$ in $X$ we have 
that $\rho_f(x,y)<\varepsilon$ if and only if
there is  $N\ge 1$ and $\sigma\in E^N$ such that $x,y\in
Z(\sigma)$. Thus the metric $\rho_f$ generates the topology on $X$. 
If $(r_{1},\dots,r_{n})$ is a contractive ratio list and $f(x)=1/r_{x_{0}}$
we recover the metric from the previous example.

We prove next that $T$ satisfies the local scaling condition 
(see Definition \ref{def:loc-scaling}) with respect to
the metric $\rho_{f}$ with $\varphi = \log f$.
First we need a lemma.

\begin{lem}\label{lem:generalized_metric}
  Let $x,y\in X$ such that their longest common prefix $\sigma$ has
  length at least $2$. Then there are points $z_i,w_i\in Z(\sigma)$,
  $i=1,2$ such that
  \begin{equation}\label{eq:ineq_distance}
  \bigl(f(z_1)f(w_1)\bigr)^\frac{1}{2} \le
  \frac{\rho_f(Tx,Ty)}{\rho_f(x,y)}\le \bigl(f(z_2)f(w_2)\bigr)^{\frac{1}{2}}.
\end{equation}
\end{lem}
\begin{proof}
  Observe that $w_\sigma=\rho_f(x,y)$ and $w_{\sigma^\prime
  }=\rho_f(Tx,Ty)$ (Equation~(\ref{eq:4})), where $\sigma^\prime
  :=\sigma_1\dots\sigma_{\vert\sigma\vert-1}$. Then
  \[
  \frac{\rho_f(Tx,Ty)}{\rho_f(x,y)}=\frac{w_{\sigma^\prime}}{w_\sigma}. 
  \]
  Note that
  if $z\in Z(\sigma)$ then $Tz\in Z(\sigma^\prime)$. Conversely, if
  $z^\prime \in Z(\sigma^\prime)$ then $\sigma_0z^\prime \in
  Z(\sigma)$.  Let $z_1,w_1\in Z(\sigma)$ the points for which the maximum in
  (\ref{eq:4}) is attained. That is,
  \[
  w_\sigma=\left(\prod_{i=0}^{\vert\sigma\vert-1}f(T^{i}z_1)\cdot\prod_{i=0}^{\vert\sigma\vert-1}f(T^{i}w_1)\right)^{-1/2}.
  \] Then $Tz_1,Tw_1\in Z(\sigma^\prime)$ and
  \begin{eqnarray*}
  w_\sigma&=&\bigl(f(z_1)f(w_1)\bigr)^{-\frac{1}{2}}\cdot
  \left(\prod_{i=1}^{\vert\sigma\vert-1}f(T^{i}z_1)\cdot\prod_{i=1}^{\vert\sigma\vert-1}f(T^{i}w_1)\right)^{-1/2}\\
  &\le&\bigl(f(z_1)f(w_1)\bigr)^{-\frac{1}{2}}\cdot w_{\sigma^\prime}.
\end{eqnarray*}
Thus the first part of (\ref{eq:ineq_distance}) holds. Let now
$\overline{z}_2,\overline{w}_2\in Z(\sigma^\prime)$ such that
\[
w_{\sigma^\prime}=\left(\prod_{i=0}^{\vert\sigma\vert-2}f(T^{i}\overline{z}_2)\cdot\prod_{i=0}^{\vert\sigma\vert-2}f(T^{i}\overline{w}_2)\right)^{-1/2}.
\] Then $z_2=\sigma_0\overline{z}_2$ and $w_2=\sigma_0\overline{w}_2$
belong to $Z(\sigma)$ and
\[
\left(\prod_{i=0}^{\vert\sigma\vert-1}f(T^{i}z_2)\cdot\prod_{i=0}^{\vert\sigma\vert-1}f(T^{i}w_2)\right)^{-1/2} 
\le w_\sigma.
\]
Thus
\[
\bigl(f(z_2)f(w_2)\bigr)^{-\frac{1}{2}} w_{\sigma^\prime}\le w_\sigma,
\] which is the second part of (\ref{eq:ineq_distance}).
\end{proof}

Let $\varepsilon >0$ be given such that $\varepsilon< \min f(x)$. Since $f$ is uniformly continuous on $X$
there 
is a $\delta >0$ such that
if $\sigma\in E^*$ is chosen so that
$1/m^{\vert\sigma\vert}<\delta$ then $\vert 
f(z)-f(w)\vert<\varepsilon$ for all $z,w\in Z(\sigma)$. Let
  $x,y\in X$ be such that the length of their common prefix $\sigma$
  satisfies $1/m^{\vert\sigma\vert}<\delta$. Suppose that $z_1,w_1$ and
  $z_2,w_2$ are the points in $Z(\sigma)$ for which the inequality
\eqref{eq:ineq_distance} holds. Since $\vert f(z_2)-f(x)\vert
  <\varepsilon$ it follows that $f(z_2)<f(x)+\varepsilon$. Similarly,
  $f(w_2)<f(x)+\varepsilon$, $f(x)-\varepsilon<f(w_1)$ and
  $f(x)-\varepsilon<f(z_1)$. Since all the above quantities are
  positive, inequality \eqref{eq:ineq_distance} implies that 
\[
f(x)-\varepsilon \le
\frac{\rho_f(Tx,Ty)}{\rho_f(x,y)}\le f(x)+\varepsilon
\]
Therefore~(\ref{eq:f_metric}) holds and (\ref{eq:0}) also holds with $\varphi = \log f$; 
hence, $T$ satisfies the local scaling condition  (by Proposition
\ref{rem:equivdef}). Since the topology on the Renault-Deaconu
groupoid $G$ associated to $(X,T)$ depends only on the topology on $X$
and not the underlying metric, if follows that it is still the case
that $C^*(G)\simeq\mathcal{O}_{n}$.
The associated automorphism $\alpha_{t}$ of $C^{*}(G)$
is given via\[
\alpha_{t}(a)(x,m-k,y)=e^{itc_{f}(x,m-k,y)}a(x,m-k,y),\]
for all $(x,m-k,y)\in G$ and $a\in C_{c}(G)$, where the cocycle
$c_{f}$ is given by\begin{equation}
  c_{f}(x,m-k,y)=\log\left(\prod_{i=0}^{m-1}f(T^{i}x)/\prod_{i=0}^{k-1}f(T^{i}y)\right).\label{eq:cocycle_gen}
\end{equation} 
These actions generalize the gauge actions on $\mathcal{O}_{n}$ described
in the previous example. To see this, recall from the previous
  example that, under our identification, $S_{j}=1_{\Gamma_{j}}$,
where $\Gamma_{j}=\{(jx,1,x)\,:\, x\in X\}$, for $j=1,\dots,n$.
If $f(x)=1/r_{x_{0}}$ an easy computation shows that \[
e^{itc_{f}(x,m-k,y)}S_{j}(x,m-k,y)=e^{-it\log r_{j}}S_{j}(x,m-k,y),\]
and we recover \eqref{eq:gauge_action}. Assuming that the Hausdorff
dimension $\beta$ of $X$ is strictly positive and finite, Theorem
\ref{thm:Main} implies that there exists a KMS-state $\omega_{\mu}$
on $\mathcal{O}_{n}$ for $\alpha$ at inverse temperature $\beta$. 

If $\log f(x)$ satisfies the Bowen condition, then $\beta$ and $\mu$
are the unique solutions of the equation $\mathcal{L}_{f,\beta}^{*}(\mu)=\mu$,
where\[
\mathcal{L}_{f,\beta}(a)(x)=\sum_{j=1}^{n}f(jx)^{-\beta}a(jx),\]
for all $x\in X$ and $a\in C(X)$ (\cite[Proposition 7.1]{Ren_ETDN_05},
\cite[Theorem 2.8 and proof of Proposition 3.5]{Ku_Re_PAMS06}). In this case, $\omega_{\mu}$
is the unique KMS state for $\alpha$ on $\mathcal{O}_{n}$ by Proposition
\ref{pro:uniqueness}.

\subsection{Graph $C^{*}$-algebras}

Suppose that $E=(E^{0},E^{1},r,s)$ is a finite directed graph, where
$r:E^1\to E^0$ and $s:E^1\to E^0$ are the range and source maps,
respectively. Suppose that $\{r_{e}\}_{e\in E^{1}}$
is a list of positive numbers\footnote{the ratio $r_e$ should not be confused with $r(e)$, the range of the edge $e$.}
 such that $r_{e}<1$ for all $e\in E^{1}$.
We say that $\{r_{e}\}_{e\in E^{1}}$ is a \emph{contractive ratio
list }for the graph $E$ (see \cite[Section 4.3]{Edg_MTFG2}). A path of length $n$ in the graph $E$
is a finite sequence $\sigma=\sigma_{0}\sigma_{1}\dots\sigma_{n-1}$
such that $\sigma_{i}\in E^{1}$ for all $i\in\{0,\dots,n-1\}$ and
$s(\sigma_{i})=r(\sigma_{i+1})$ for all $i\in\{0,\dots,n-2\}$. Note
that we adopt the recent convention from the theory of graph
$C^*$-algebras when we define a path in a graph; namely, we switch the
traditional role of $r$ and $s$ (see, for example, \cite{Rae_CBMS05} 
for more details about this). We
write $E^{n}$ for the set of paths of length $n$ and $E^{*}=\bigcup_{n}E^{n}$
for the set of \emph{finite paths} in the graph $E$. We extend the
definition of $r$ and $s$ to $E^{*}$ via $r(\sigma)=r(\sigma_{0})$
and $s(\sigma)=s(\sigma_{n-1})$ if $\sigma\in E^{n}$. The \emph{infinite
path space} $X$ is defined via\[
X:=E^{\infty}:=\{(x_{n})_{n\in\mathbb{N}}\,:\, x_{n}\in E^{1}\text{ and }s(x_{n})=r(x_{n+1})\}.\]
For an element $x\in E^{\infty}$ we set $r(x)=r(x_{0})$. We say
that the graph $E$ is \emph{irreducible }(or strongly connected)
if for every two elements $v,w\in E^{0}$ there is a path $\sigma\in E^{*}$
such that $r(\sigma)=v$ and $s(\sigma)=w$. For $v\in E^0$ we write
$vE^\infty$ for the set of infinite sequences $x\in X$ such that
$r(x)=v$. Similarly  one can define $vE^k$ and $vE^kw$. 

If $s$ is a positive real number, then $s$-dimensional \emph{Perron
numbers} for the graph $E$ are positive numbers $q_{v}$, one for
each vertex $v\in E^{0}$, such that\begin{equation}
q_{v}^{s}=\sum_{w\in E^{0},e\in vE^{1}w}r(e)^{s}q_{w}^{s}\label{eq:Hausdorff_graph}\end{equation}
for all $v\in E^{0}$ (see, for example, \cite[Section 6.6]{Edg_MTFG2}). For an irreducible
graph $E$ there is a unique number $s\ge0$ such that Perron numbers
exist (\cite[Theorem 6.9.6]{Edg_MTFG2}).

If $\sigma\in E^*$ then the cylinder $Z(\sigma)$ is
\[
Z(\sigma)=\{x\in E^{\infty}\,:\, x_{i}=\sigma_{i},i=0,\dots,n-1\}.\]
Using \cite[Proposition 2.6.5]{Edg_MTFG2}, one can show that there is
a metric $\rho$ on $X$ such that the diameter of $Z(\sigma)$ equals
$r_{\sigma}q_{s(\sigma)}$, for all $\sigma\in E^*$. It follows that
the metric $\rho$ satisfies the property
\[
\rho(ex,ey)=r_{e}\rho(x,y)\text{ for all }x,y\in s(e)E^{\infty}.\]
Provided that the number $s$ which satisfies condition \eqref{eq:Hausdorff_graph}
is positive, the Hausdorff dimension $\beta$ of $X$ equals $s$
and the Hausdorff measure on $X$ is the unique Borel measure such
that $\mu^{\beta}(Z(\sigma))=q_{s(\sigma)}^{\beta}r_{\sigma}^{\beta}$
for all $\sigma\in E^{*}$ (\cite[Section 6.6]{Edg_MTFG2}).

The left shift $T:X\to X$ defined by $T(x_{n})=(x_{n+1})$ is a local
homeomorphism on $X$.  We assume that the graph
$E$ is irreducible and it satisfies condition (L) from \cite{KuPaRa_PJM98}
(or, equivalently since $E$ is finite, condition (I) from \cite{Cu_Kr_80}).
That is, we assume that every loop in $E$ has an exit. Under these
assumptions
if $G$ is the  Renault-Deaconu groupoid
associated to the pair $(X,T)$,
then $C^{*}(G)$ is isomorphic to the graph $C^{*}$-algebra $C^{*}(E)$
(\cite{KuPaRaRe_JFA97},\cite{KuPaRa_PJM98}; see also \cite{Rae_CBMS05}).
The local homeomorphism $T$ satisfies the local scaling property
with\[
\lim_{y\to x}\frac{\rho(Tx,Ty)}{\rho(x,y)}=\frac{1}{r_{x_{0}}},\]
that is, $\varphi(x)=-\log r_{x_{0}}$. The corresponding automorphism
of $C^{*}(G)$ is defined then via\[
\alpha_{t}(P_{v})=P_{v}\;\text{for all }v\in E^{0}\]
and\[
\alpha_{t}(S_{e})=e^{-it\log r_{e}}S_{e},\]
where $(\{P_{v}\}_{v\in E^{0}},\{S_{e}\}_{e\in E^{1}}\}$ is a \emph{universal}
Cuntz-Krieger family generating $C^{*}(G)$ (\cite{KuPaRa_PJM98,Rae_CBMS05}).

Theorem \ref{thm:Main} and \ref{pro:uniqueness} imply that there
exists a KMS-state for $\alpha$ at inverse temperature $\beta$ if
and only if $\beta$ is the unique positive number for which Perron
numbers exist. In this case, there is a unique $(\alpha,\beta)$-KMS
state, namely\[
\omega_{\mu}(f)=\int E(f)d\mu,\]
where $\mu$ is the normalized Hausdorff measure on $X$ and $E$ is the
canonical expectation onto $C(X)$. 

If $r_{e}=r\in(0,1)$ for all $e\in E^{1}$, then the Hausdorff dimension
satisfies $\lambda=r^{-\beta}$, where $\lambda$ is the Perron-Frobenius
eigenvalue of the vertex matrix $A_{E}$ of the graph $E$ (see, for
example, \cite[Theorem 6.9.6]{Edg_MTFG2}). Proposition
\ref{pro:Entropy} implies that the topological entropy of $T$ is
$\log\lambda$.

\subsection{Generalized gauge actions on graph $C^{*}$-algebras}

The construction of Example \ref{sub:Generalized-gauge-actions-Cuntz}
extends to the graph $C^{*}$-algebras described in the previous example.
Suppose that $X$ is the infinite path space of a finite directed
graph $E=(E^{0},E^{1},r,s)$ which is irreducible and satisfies condition
(L) of \cite{Ku_Re_PAMS06}. We endow $X$ with the product topology
so that it is a compact topological space and we let $T$ be the left
shift map on $X$ as in the previous example. Fix a continuous function $f$ on $X$ such that
$f(x)>1$ for all $x\in X$ and fix a list of positive number
$\{q_v\}_{v\in E^0}$ such that
\begin{equation}
  \label{eq:5}
  f(x)>\frac{q_{s(x_0)}}{q_{r(x_0)}}\;\text{ for all}\;x\in E^\infty.
\end{equation}
For example, if $q_v=1$ for all $v\in V$ then the above condition is
trivially satisfied. We define  the diameter of  a cylinder
$Z(\sigma)$, $\sigma\in vE^{*}$, to be $w_\sigma\cdot
q_{s(\sigma)}$, where 
\begin{equation}\label{eq:wsigma_graph}
w_{\sigma}:=\max_{z,w\in Z(\sigma)}\left(\prod_{i=0}^{\vert\sigma\vert-1}f(T^{i}z)\cdot\prod_{i=0}^{\vert\sigma\vert-1}f(T^{i}w)\right)^{-1/2}.
\end{equation}
Since $f(x)>1$ and the numbers $q_v$ satisfy (\ref{eq:5}), it follows
that if $\alpha,\beta\in E^*$ are such that $\alpha<\beta$  then $\diam
Z(\alpha)<\diam Z(\beta)$ and $\lim_{n\to \infty}\diam Z(\sigma|_n)=0$
for all $\sigma\in E^{\infty}$.  Then \cite[Proposition
2.6.5]{Edg_MTFG2} implies  that there
is a metric $\rho_{f}$ on $vX$ such
that the diameter of $Z(\sigma)$ with respect to $\rho_f$ equals  $w_\sigma\cdot
q_{s(\sigma)}$ for all $\sigma\in vE^\infty$, for each
$v\in E^{0}$. Namely, if 
$x,y\in vE^\infty$ and $\sigma\in E^*$ is their longest common prefix,
then
\[
\rho_f(x,y)=w_\sigma\cdot q_{s(\sigma)}.
\]
Since
$X$ is the finite disjoint union of $vX$ we can extend $\rho_{f}$
to a metric on $X$. 

If $f(x)=1/r_{x_{0}}$ for a contractive ratio
list $\{r_{e}\}_{e\in E^{1}}$ and $\{q_v\}_{v\in E^{0}}$ are $s$-dimensional
Perron numbers for the graph $E$ then the condition (\ref{eq:5}) is
implied by (\ref{eq:Hausdorff_graph}).  Thus we recover the metric
from the previous example. 

One can easily see that if $x$ and $y$ have a common prefix $\sigma$
of length at least $2$ then
\[
\frac{\rho(Tx,Ty)}{\rho(x,y)}=\frac{w_{\sigma^\prime}}{w_\sigma},
\] because $s(\sigma^\prime)=s(\sigma)$, where $\sigma^{\prime}$ is
obtained from $\sigma$ by removing the first entry
($\sigma^\prime=\sigma_1\dots \sigma_{\vert \sigma\vert -1}$). Then the analogue 
of Lemma \ref{lem:generalized_metric} holds and  $T$ satisfies the local
scaling property  (by Proposition \ref{rem:equivdef}) with respect to 
the metric $\rho_{f}$ with\[
\lim_{y\to x}\frac{\rho(Tx,Ty)}{\rho(x,y)}=f(x),\]
for all $x\in X$, and condition (\ref{eq:0}) holds with
$\varphi =\log f$. Assuming that the graph $E$ is irreducible and
  satisfies condition [L] of \cite{KuPaRa_PJM98} , the $C^*$-algebra $C^*(G)$ of the
  Renault-Deaconu groupoid $G$ associated to $(X,T)$ is still isomorphic
  to the graph $C^*$-algebra $C^*(E)$. The associated 
action $\alpha_{t}$ on $C^{*}(G)$ is given via\[
\alpha_{t}(a)(x,m-n,y)=e^{itc_{f}(x,m-n,y)}a(x,m-n,y),\]
for all $(x,m-n,y)\in G$ and $a\in C_c(G)$, where the cocycle $c_{f}$
is defined as in \eqref{eq:cocycle_gen}. These actions generalize
the gauge actions described in the previous example. Indeed, a universal
Cuntz-Krieger family on $C^{*}(E)$ is given by $\bigl(\{P_v\}_{v\in E^{0}},\{S_e\}_{e\in E^{1}}\bigr)$,
where $P_v=1_{\Delta_{v}}$ and $S_e=1_{\Gamma_{e}}$, and \[
\Delta_{v}=\{(x,0,x)\in G^{0}\,:\: r(x)=v\}\]
and\[
\Gamma_{e}=\{(ex,1,x)\in G\,:\, x\in s(e)E^{\infty}\}.\]
Then an easy computation shows that\begin{eqnarray*}
e^{itc_{f}(x,m-n,y)}1_{\Delta_{v}}(x,m-n,y) & = & 1_{\Delta_{v}}(x,m-n,y),\;\text{and}\\
e^{itc_{f}(x,m-n,y)}1_{\Gamma_{e}}(x,m-n,y) & = & e^{it\log r_{e}}1_{\Gamma_{e}}(x,m-n,y).\end{eqnarray*}

Let $\beta$ be the Hausdorff dimension of $X$ with respect to the
metric $\rho_{f}$ and assume that $0<\beta<\infty$. Theorem \ref{thm:Main}
implies that there exists a KMS state $\omega_{\mu}$ on the graph
$C^{*}$-algebra $C^{*}(E)$ for $\alpha$ at inverse temperature
$\beta$. Moreover, if $\log f$ satisfies the Bowen condition, then
$\beta$ is the unique inverse temperature for which KMS states exist
and $\omega_{\mu}$ is the unique $(\alpha,\beta)$-KMS state. The
Hausdorff dimension $\beta$ and the measure $\mu$ are the unique
solutions of the equation $\mathcal{L}_{f,\beta}^{*}(\mu)=\mu$ (\cite{Ren_ETDN_05,Ku_Re_PAMS06}),
where\[
\mathcal{L}_{f,\beta}(a)(x)=\sum_{e\in E^{1},s(e)=r(x)}f(ex)^{-\beta}a(ex).\]
We recover, thus, some of the results of \cite{Exe_BBMS04}.

\subsection{Coverings of $\mathbb{T}$}

Let $f:[0,1]\to\mathbb{R}$ be a continuous  positive function such
that $f(0)=f(1)$ 
and $n:=\int_{0}^{1}f(t)dt$ is a positive integer such that $n\ge2$.
For example, the constant function $n$ satisfies these requirements.
Let $X=\mathbb{T}=\mathbb{R}/\mathbb{Z}$ and let $\rho$ the induced metric
on $X$. We define $T:X\to X$ by\[
T(x)=\int_{0}^{x}f(t)dt.\]
Then $T$ is a local homeomorphism. Moreover, $T$ is an $n$-fold
covering map. One can easily compute that\[
\lim_{y\to x}\frac{\rho(Tx,Ty)}{\rho(x,y)}=f(x).\]
Thus $T$ satisfies the local scaling condition and $\varphi(x)=\log f(x)$.
Since the Hausdorff dimension $\beta=1$, Theorem \ref{thm:Main}
implies that $\omega_{\mu}$ is an $(\alpha,1)$-KMS state, where
$\mu=\mu^{1}$. If $f$ is continuously differentiable and $f(x)>1$,
then $\varphi=\log f$ is both positive and H\"older. Thus $c_{\varphi}^{-1}(0)$
is principal and $\varphi$ satisfies the Bowen condition. Moreover,
$T$ is positively expansive and exact. Therefore, by Proposition
\ref{pro:uniqueness}, $\omega_{\mu}$ is the unique KMS state on
$C^{*}(G)$.

Suppose now that $f(x)=n$ for all $x\in[0,1]$. Then $f$ is continuously
differentiable and $f(x)>1$. Moreover, Proposition \ref{pro:Entropy}
implies that $h(T)=\log n$.

\subsection{The Sierpinski octafold}

Consider the triangle $Y$ with vertices $v_{1}=(1,0,0)$, $v_{2}=(0,1,0)$,
and $v_{3}=(0,0,1)$ in $\mathbb{R}^{3}$. Let $f_{1},f_{2},f_{3}$
be the the restrictions to $Y$ of the linear maps defined by the 
matrices
\[
A_{1}=\left[\begin{array}{ccc}
1 & 1/2 & 1/2\\
0 & 1/2 & 0\\
0 & 0 & 1/2\end{array}\right],\, A_{2}=\left[\begin{array}{ccc}
1/2 & 0 & 0\\
1/2 & 1 & 1/2\\
0 & 0 & 1/2\end{array}\right],\, A_{3}=\left[\begin{array}{ccc}
1/2 & 0 & 0\\
0 & 1/2 & 0\\
1/2 & 1/2 & 1\end{array}\right].
\]
Note that for $i = 1, 2, 3$, we have $f_i(Y) \subset Y$  and
 $f_{i}$ is a similarity with ratio $1/2$.
So $(f_{1},f_{2},f_{3})$ is an iterated function system (\cite{Hut_81})
and it admits a unique invariant set $K$ (\cite{Hut_81,Edg_MTFG2})\[
K=f_{1}(K)\bigcup f_{2}(K)\bigcup f_{3}(K).\]
The set $K$ is a copy of the Sierpinski gasket. Using the terminology from
\cite{Kig_CUP01} and \cite{Str_Prin06}, we call the sets $f_{i}(K)$,
$i=1,2,3$, $1$-cells.

Now consider the octahedron in $\mathbb{R}^{3}$ with vertices at
$(\pm 1, 0, 0)$, $(0, \pm 1,  0)$ and $(0, 0, \pm 1)$; note that each face
is isometric to $Y$ and $Y$ is one of the faces. We consider four copies
of the Sierpinski gasket $K$  (one of which is $K$ itself)
on alternating faces of the octahedron; so any two of the Sierpinski 
gaskets intersect only at a single point. 
Let $X$ be the union of the four Sierpinski gaskets.
Note that each point in $X$ has a neighborhood which is similar to
a neighborhood of the Sierpinski gasket. Thus $X$ is what is called
in \cite{Str_TAMS03} a \emph{fractafold} (see also \cite{Str_Prin06} and
\cite{Str_CJM98}). We call $X$ the \emph{Sierpinski octafold}.
One can also think of the octafold as four copies of the Sierpinski
gasket such that any two of them are glued at one of the vertices
as in Figure \ref{Flo:figure}, where the dotted curves mean that the
two endpoints of the curve are identified.

\begin{figure}
\caption{The Sierpinski octafold}
\label{Flo:figure}

\includegraphics[scale=0.5]{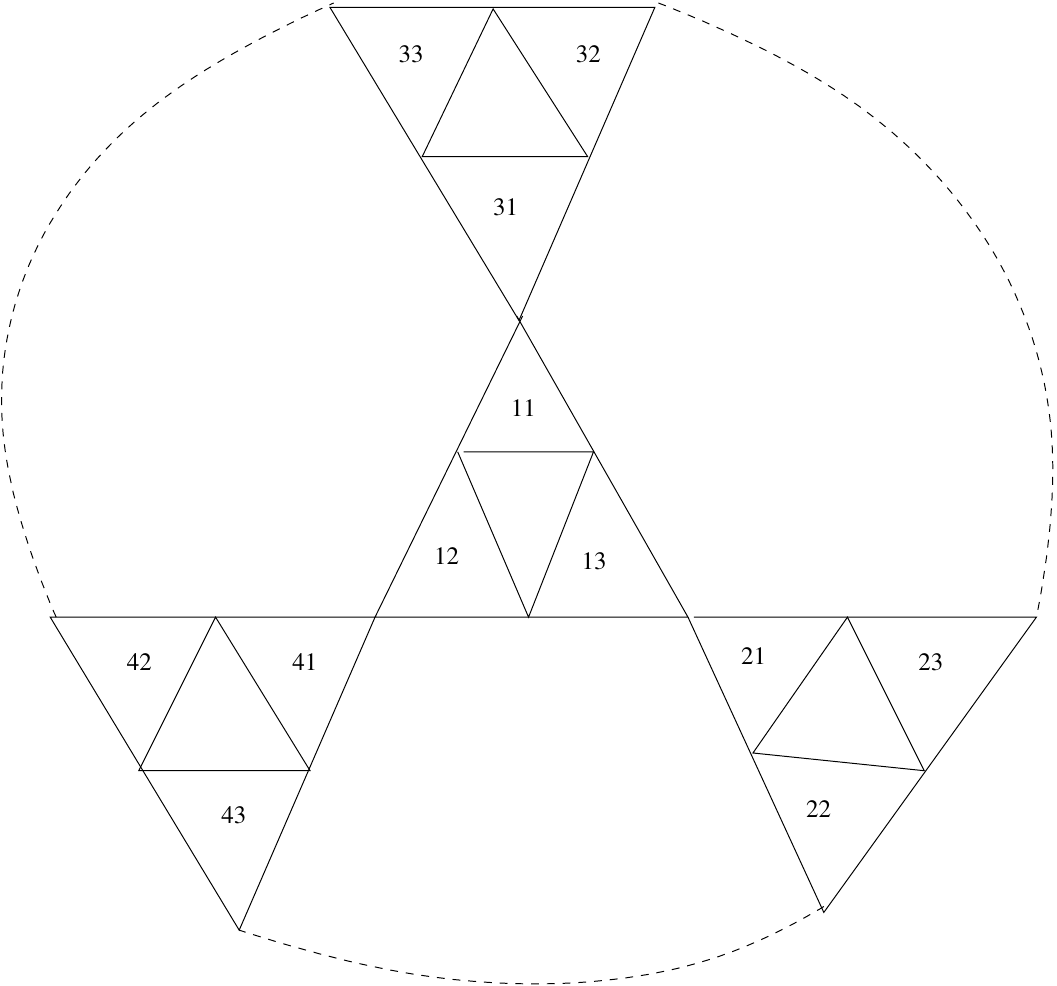}

\end{figure}
The metric $\rho$ on $X$ is the restriction of the Euclidean metric
to $X$. Therefore the Hausdorff dimension of $K$ is
  $\log3/\log2$ (see, for example, \cite[Theorem
6.5.4]{Edg_MTFG2}). Since $X$ is the union of four copies of
$K$ and the Hausdorff dimension of a union of sets equals the maximum of the
Hausdorff dimension of each of the sets (\cite[Theorem
6.1.7]{Edg_MTFG2}) it follows that the Hausdorff dimension of $X$
is also $\log3/\log2$.
To describe the local homeomorphism, label the four copies of the
Sierpinski gasket, which we will call $0$-cells, as $1,2,3,$ and
$4$. Let $11,12,13$ be the three $1$-cells of the $0$-cell $1$,
and similarly for the other three $0$-cells (see Figure \ref{Flo:figure}).
The local homeomorphism $T$ we define is uniquely determined by the
following two properties
\begin{enumerate}
\item The vertices of the $0$-cells are left fixed.
\item The restriction to each $1$-cell defines an affine homeomorphism
onto the $0$-cell that shares exactly one vertex with the original
$1$-cell.
\end{enumerate}
%
{}It is enough to indicate where the points that are the vertices of
the $1$-cells and are not vertices of the $0$-cells map. We call
these points \emph{midpoints}. Each midpoint may be written $.5(\epsilon_{1}e_{\sigma(1)}+\epsilon_{2}e_{\sigma(2)})$
where $\epsilon_{i}=\pm1$, $e_{i}$ is a canonical basis element
of $\mathbb{R}^{3}$ and $\sigma$ is in the permutation group $S_{3}$.
Then \[
T\bigl(.5(\epsilon_{1}e_{\sigma(1)}+\epsilon_{2}e_{\sigma(2)})\bigr)=-\epsilon_{1}\epsilon_{2}e_{\sigma(3)}.\]
It is easy to see that this is well defined and $T$ is a local
homeomorphism
such that
\[
\lim_{y\to x}\frac{\rho(Tx,Ty)}{\rho(x,y)}=2.\]
It does not, however, satisfy condition (\ref{eq:0}) of Proposition 
\ref{rem:equivdef} and, thus, it does not satisfy the local scaling
condition. To see this, consider a point $y$ in the $1$-cell $11$ and
another point $z$ in the $1$-cell $12$ such that they lie on the line
segments, that form an angle of $\pi/3$ and intersect at the midpoint 
$x$ (where  the two $1$-cells meet).   We may also assume  that they 
are equidistant from the midpoint $x$ (see Figure \ref{Flo:figure}).
Their images lie on adjacent edges of a square, so the Pythagorean Theorem 
implies that $\rho(Ty, Tz)=2\sqrt{2}\rho(y,z)$, while $\rho(Tx,Ty)=2\rho(x,y)$ 
and $\rho(Tx,Tz)=2\rho(x,z)$. Since we can take $y$ and $z$ as close as we
want to $x$, it follows that $T$ does not satisfy the local scaling
property. 

We claim, however, that the conclusion of Theorem \ref{thm:Main} is still valid 
for this example with $\varphi=\log 2$. Since the proofs of Proposition 
\ref{pro:uniqueness} and Proposition \ref{pro:Entropy} do not depend on 
the local scaling property and only on the existence of the KMS-state 
$\omega_\mu$, it follows that they are also valid, once we prove the claim. 

Let $G$ be the associated Renault-Deaconu groupoid. The action on
$C^{*}(G)$ is given via\[
\alpha_{t}(f)(x,m-n,y)=e^{-it(m-n)\log2}f(x,m-n,y).\]
Recall from \cite[Proposition
II.5.4]{Ren_LNM793} that in order for the Hausdorff measure $\mu$ to be a quasi-invariant
measure for $G$ the equality (\ref{eq:RN}) need only hold 
almost everywhere. One can modify the proof of Theorem \ref{thm:Main}
to show that this is the case for this example. Let $C$ be any of the $1$-cells of the
octafold. Then the restriction of $T$ to $C$ is a similarity with ratio
$2$. Theorem 6.1.9 of \cite{Edg_MTFG2} implies that $\mu(TU)=2^\beta\mu(U)$ for
all subsets of $C$, where $\beta=\log 3/\log 2$ is the Hausdorff dimension. In
particular the equality is true for all open sections $U$ of $C$. 
It follows
that the equation (\ref{eq:RN}) is true with $\varphi =\log 2$ for all points 
except possibly the 
midpoints and the conclusion of Theorem \ref{thm:Main} is true
for the octafold.

Thus,  $\beta=\log3/\log2$
is the unique inverse temperature which admits a KMS state on $C^{*}(G)$
for $\alpha$, the $(\alpha,\beta)$-KMS state $\omega_{\mu}$ given
by \eqref{eq:KMSstate} is unique, and the topological entropy is $h(T)=\log3$.

\bibliographystyle{amsplain}
\bibliography{bib_hmkms}

\end{document}